\documentclass{article}
\usepackage[english]{babel}
\usepackage[latin1]{inputenc}
\usepackage[T1]{fontenc}
\usepackage{amsmath,amssymb,amsthm,xspace,enumerate,url}
\usepackage[hypertex]{hyperref}
\usepackage{svn}

\title{Computable functions of reals}

\author{Katrin Tent and Martin Ziegler}
\date{19. August, 2009}

\newtheorem{satz}{Theorem}[section]
\newtheorem{lemma}[satz]{Lemma}
\newtheorem{proposition}[satz]{Proposition}
\newtheorem{folgerung}[satz]{Corollary}
\newtheorem{definition}[satz]{Definition}
\newtheorem{bemerkung}[satz]{Remark}

\newcommand{\nc}{\newcommand}
\nc{\sa}{semialgebraic\xspace}
\nc{\el}{elementary\xspace}
\nc{\low}{lower \el}
\nc{\inv}[1]{\frac{1}{#1}}
\nc{\N}{\mathbb{N}}
\nc{\Np}{\N_{\scriptscriptstyle >0}}
\nc{\Z}{\mathbb{Z}}
\nc{\Q}{\mathbb{Q}}
\nc{\R}{\mathbb{R}}
\nc{\Rp}{\R_{\scriptscriptstyle >0}}
\nc{\C}{\mathbb{C}}
\nc{\F}{\ensuremath{\mathcal{F}}\xspace}
\nc{\A}{\mathcal{A}}
\nc{\E}{\mathbb{E}}
\nc{\Epsilon}{{\Large $\epsilon$}} 
\nc{\ap}{approximable\xspace}
\nc{\e}{\mathrm{e}}
\nc{\ii}{\,\mathrm{i}}
\nc{\Es}{\E\setminus\R_{\scriptscriptstyle\leq0}}

\DeclareMathOperator{\h}{h}
\DeclareMathOperator{\di}{dist}

\DeclareMathOperator{\dd}{d\hspace{-1pt}}
\DeclareMathOperator{\vol}{vol}
\DeclareMathOperator{\D}{D}

\begin{document}

\maketitle

\begin{abstract} We introduce a new notion of
  computable function on $\R^N$ and prove some basic properties.  We
  give two applications, first a short proof of Yoshinaga's theorem
  that periods are \el (they are actually \low). We also show that the
  \low complex numbers form an algebraically closed field closed under
  exponentiation and some other special functions. \end{abstract}
\section{Introduction}

We here develop a notion of computable functions on the reals along
the lines of the \emph{bit-model} as described in
\cite{Braverman_Cook}. In contrast to the algebraic approach towards
computation over the reals developed in \cite{BSS}, our approach goes
back to Grzegorczyk and the hierarchy of \el functions and real
numbers developed by him and Mazur (see \cite{Grz} footnote p.\ 201).
We thank Dimiter Skordev for a careful reading of an earlier version
of this note. We also thank the referee for pointing out several
flaws in the first version.
\section{Good classes of functions}

A class \F of functions $\N^n\to\N$ is called \emph{good} if it
contains the constant functions, the projection functions, the
successor function, the modified difference function
$x\dot{-}y=\max\{0,x-y\}$, and is closed under composition and
\emph{bounded summation}
\[f(\bar x,y)=\sum_{i=0}^yg(\bar
x,i).\] The class of \emph{\low} functions is the smallest
good class. The smallest good class which is also closed under
\emph{bounded product}
\[f(\bar x,y)=\prod_{i=0}^yg(\bar x,i),\]
or -- equivalently -- the smallest good class which contains $n\mapsto
2^n$, is the class of \emph{\el} functions The \el functions are the
third class \Epsilon$^3$ of the Grzegorczyk hierarchy. The \low
functions belong to \Epsilon$^2$. It is not known whether all
functions in \Epsilon$^2$ are \low.

A function $f:\N^n\to\N^m$ is an \F--function if its components
$f_i:\N^n\to\N$, $i=1,\ldots,m$, are in \F.  A relation
$R\subset\N^n$ is called an \F--relation if its characteristic
function belongs to \F. Note that a good class is closed under
the bounded $\mu$--operator: if $R$ belongs to \F, then so does the
function
\[f(\bar x,y)=\min\{i\mid R(\bar x,i)\lor i=y\}.\]
As a special case we see that $\lfloor\frac{x}{y}\rfloor$ is \low.
The \F--relations are closed under Boolean combinations and bounded
quantification:
\[S(x,y)\;\Leftrightarrow\;\exists\,i\leq y\, R(x,i).\]
It follows for example that for any $f$ in \F the maximum function
\[\max_{j\leq y}f(\bar x,j)=\min\{\,i\mid \forall j\leq y \; f(\bar
  x,j)\leq i\,\}\]
is in \F since it is bounded by $\sum_{i=0}^yf(\bar x,i)$.\\

We call a set $X$ an \F--retract (of $\N^n$) if there are functions
$\iota:X\to\N^n$ and $\pi:\N^n\to X$ given with
$\pi\circ\iota=\mathrm{id}$ and $\iota\circ\pi\in\F$.  Note that the
product $X\times X'$ of two \F--retracts $X$ and $X'$ is again an
\F--retract (of $\N^{n+n'}$) in a natural way. We define a function
$f:X\to X'$ to be in \F if $\iota'\circ f\circ\pi:\N^n\to\N^{n'}$ is
in \F. By this definition the two maps $\iota:X\to\N^n$ and
$\pi:\N^n\to X$ belong to \F. For an \F--retract $X$, a subset of $X$
is in \F if its characteristic function is. It now makes sense to say
that a set $Y$ (together with $\iota:Y\to X$ and $\pi:X\to Y$) is a
retract of the retract $X$. Clearly $Y$ is again a retract in a
natural way.

Easily, $\N_{>0}$ is a \low retract of $\N$ and $\Z$ is a \low retract
of $\N^2$ via $\iota(z)=(\max(z,0),-\min(0,z))$ and $\pi(n,m)=n-m$. We
turn $\Q$ into a \low retract of $\Z\times\N$ by setting
$\iota(r)=(z,n)$, where $\frac{z}{n}$ is the unique representation of
$r$ with $n>0$ and $(z,n)=1$. Define $\pi(z,n)$ as $\frac{z}{n}$ if
$n>0$ and as $0$ otherwise.

For the remainder of this note we will consider $\Z$ and $\Q$ as
\emph{fixed \low} retracts of $\N$, using the maps defined in the last
paragraph.\\

The following lemma is clear.
\begin{lemma} Let $X$, $X'$ and $X''$ be \F--retracts, and $f:X\to
  X'$ and $f':X'\to X''$ in \F. Then $f'\circ f$ also belongs to
  \F. \qed
\end{lemma}

The \emph{height} $\h(a)$ of a rational number $a=\frac{z}{n}$, for
relatively prime $z$ and $n$ ($n>0$), is the maximum of $|z|$ and
$n$. The height of a tuple of rational numbers is the maximum height
of its elements.

\begin{lemma}\label{l:q-beschraenkt}
  If $f:\Q^N\to\Q$ is in \F, there is an \F--function $\beta_f:\N\to\N$
  such that
  \[\h(a)\leq e\;\Rightarrow\;|f(a)|\leq \beta_f(e).\]
\end{lemma}
\begin{proof}
  Write $f(\frac{z_1}{n_1},\ldots,\frac{z_N}{n_N})=g(z_1,\ldots,n_N)$
  for an \F--function $g:(\Z\times\N)^N\to\Q$.  The function $\lceil
  x\rceil:\Q\to\N$ is \low. So we can define $\beta_f(e)$ to be the
  maximum of all $\lceil |g(x_1,\ldots,x_{2N})|\rceil$, where
  $|x_i|\leq e$. $\beta_f$ is easily seen to be in \F.
\end{proof}

\begin{lemma}\label{l:q-add}
  Let $g:X\times\N\to\Q$ be in \F for some \F-retract $X$. Then there is an \F--function
  $f:X\times\N\times\N_{>0}\to\Q$ such that
  \[\Bigl|f(x,y,k)-\sum_{i=0}^yg(x,i)\Bigr|<\inv{k} \mbox{\ for all\ }
  x\in X, y\in\N \mbox{\ and\ } k\in\N_{>0}.\]
\end{lemma}
\begin{proof}
  We note first that for every \F--function $t:X\times\N\to\Z$ the
  function $\sum_{i=0}^yt(x,i)$ belongs to \F. It is also easy to see
  that there is a \low function $h:\Q\times\N\to\Z$ such that
  \[\Bigl|\frac{h(r,k)}{k}-r\Bigr|<\inv{k}.\]
  Now define $f$ by
  \[f(x,y,k)=\frac{\sum_{i=0}^yh(g(x,i),(y+1)k)}{(y+1)k}\]
\end{proof}

\begin{definition} A real number $x$ is an \F--real if for some
\F--function $a:\N\to\Q$ \[|x-a(k)|<\inv{k}\] \end{definition}

\begin{lemma}\label{l:f-real}
$x$ is an \F--real if and only if there is an \F--function
$n:\N\to\Z$ with  $\frac{n(k)-1}{k}<x<\frac{n(k)+1}{k}$.\qed
\end{lemma}

\begin{bemerkung}\label{b:low-decimals}
  Let $x=\sum_{i=0}^\infty a_ib^{-i}$ for integers $b\geq 2$ and
  $a_i\in\{0,1,\ldots,b-2\}$. Then $x$ is an \F-real if and only if
  \[a_i=f(b^i)\]
  for some \F--function $f$.
\end{bemerkung}

In \cite{skordev} Skordev has shown among other things that $\pi$ is
\low and that $\e$ is in \Epsilon$^2$. Weiermann \cite{weiermann_e}
proved that $\e$ is \low. He used the representation $\e=\sum\inv{n!}$
and a theorem of d'Aquino \cite{daquino}, which states that the graph
of the factorial is $\Delta_0$--definable and therefore
\low.\footnote{D. Skordev and also the referee have informed us that
  lowness of the graph of the factorial follows from the fact that the
  class of functions with \low graph is closed under bounded
  multiplication, which is not hard to prove.}

We will show in Section \ref{sec:periods} that volumes of bounded
$0$--definable \sa sets are \low, and so is $\pi$ as the volume of the
unit circle. In Section \ref{sec:ift} we show that the exponential
function maps \low reals into \low reals.

\section{Functions on $\R^N$}

Let $O$ be an open subset of $\R^N$ where we allow $N=0$. For $e\in \Np$
put
\[O\restriction e=\bigl\{x\in O\bigm| |x|\leq e\bigr\}\]
and
\[O_e=\Bigl\{x\in O\restriction e\Bigm|
\di(x,\R^N\setminus O)\geq\inv{e}\Bigr\}.\]
(We use the maximum norm on $\R^N$.)
Notice:

\begin{enumerate}
\item $O_e$ is compact.
\item $U\subset O\;\Rightarrow \;U_e\subset O_e$.
\item $e<e'\;\Rightarrow \;O_e\subset O_{e'}$.
\item $O=\bigcup_{e\in\N}O_e$
\item $O_e\subset (O_{2e})^\circ_{2e}$
\end{enumerate}

\vspace{1em}
\begin{definition}\label{d:f-function}
  Let \F be a good class.
  A function $F:O\to\R$ is in \F if there are \F--functions $d:\N\to\N$
  and $f:\Q^N\times\N\to\Q$ such that for all $e\in\Np$ and all
  $a\in\Q^N$ and  $x\in O_e$
  \begin{equation}\label{e:f-function}
    |x-a|<\inv{d(e)} \to |F(x)-f(a,e)| < \inv{e}.
  \end{equation}
  If \eqref{e:f-function} holds for all $x\in O\restriction e$, we
  call $F$ uniformly \F.\footnote{As pointed out by the referee even
    in the case when \F is the class of \low functions, the
    \F--functions on $\R$ are not necessarily computable in the sense
    of \cite{weihrauch}.}
\end{definition}
We will always assume that $d:N\to N$ is strictly increasing.

\begin{bemerkung}\label{b:extends-uniform}
  Clearly, a function is uniformly \F on $O$ if it can be extended to
  an \F--function on some $\epsilon$-neighborhood of $O$.
\end{bemerkung}

This definition easily extends to $f:O\to\R^M$ (again under the
maximum norm). Then $f$ is in \F if and only if all $f_i$,
$i=1,\ldots,M$, are in \F.
\begin{lemma}\label{l:constant}
  \F--functions  map \F--reals to \F--reals.
  A constant function on $\R^N$  is uniformly in
  \F if and only if its value is an  \F--real. \qed
\end{lemma}

\begin{definition}
  We call a function $F:O\to\R^M$ (uniformly) \F-bounded, if there is
  an \F--function $\beta_F:\N\to\Q$ in \F such that $|F(x)|\leq
  \beta_F(e)$ for all $x\in O_e$ (\/$x\in O\restriction e$,
  respectively).
\end{definition}
Note that $\exp:\R\to\R$ is not \low bounded, but elementary bounded.
\begin{lemma}\label{l:bounded}
  If $F:O\to\R^M$ is (uniformly) in \F, then $F$ is (uniformly)
  \F-bounded.
\end{lemma}
\begin{proof}
  Assume the \F--functions $d$ and $f$ satisfy \eqref{e:f-function}.
  Fix a number $e$ and consider an $x\in O_e$. Choose $z\in\Z^N$
  such that the distance between $x$ and $a=\inv{d(e)}z$ is less
  than $\inv{d(e)}$. Since $|a|<e+1$, the height of $a$ is smaller
  than $(e+1)d(e)$. By Lemma~\ref{l:q-beschraenkt} we have $|f(a,e)|\leq
  \beta_f((e+1)d(e))$, and therefore
  \[|F(x)|\leq\beta_f((e+1)d(e))+1=\beta_F(e).\]
\end{proof}
\begin{lemma}\label{l:stetig}
  \F--functions $O\to\R$ are continuous.
\end{lemma}
\begin{proof}
  We show that $F$ is uniformly continuous on every $O_e$.  Assume that
  $d$ and $f$ satisfy \eqref{e:f-function}. It suffices to show that for all
  $x,x'\in O_e$
  \[|x-x'|<\frac{2}{d(e)} \to |F(x)-F(x')| < \frac{2}{e}.\]
  Assume $|x-x'|<\frac{2}{d(e)}$. Choose $a\in\Q^N$ such that
  $|x-a|<\inv{d(e)}$ and $|x'-a|<\inv{d(e)}$. Then both $F(x)$ and
  $F(x')$ differ from $f(a,e)$ by less than $\inv{e}$. Whence
  $|F(x)-F(x')| < \frac{2}{e}$.
\end{proof}

\begin{lemma}\label{l:comp-u}
   If $F:O\to \R^M$ is in \F, $U\subset\R^M$ open and $G:U\to\R$
   uniformly in \F, then $G\circ F:F^{-1}U\cap O\to\R$ is in \F. If $F$
   is uniformly in \F, then so is $G\circ F$.
\end{lemma}
\begin{proof}
  Assume that $F$ satisfies \eqref{e:f-function} with the
  \F--functions $d$ and $f$ and assume that $G$ satisfies
  \eqref{e:f-function} with $d'$ and $g$.

  Let $\beta=\beta_F$ be as in \ref{l:bounded} and set
  $V=F^{-1}(U)\cap O$. Clearly we may assume $\beta(e)\geq e$ for all $e\in\N$.
  So if $x\in V_e\subseteq O_e$, then $F(x)\in U\restriction \beta(e)$.
  Thus for all $e\in\N$, $a\in\Q^N$ and $x\in V_e$ we have
  \begin{gather*}
    |x-a|<\inv{d(d'(\beta(e)))} \Rightarrow
    |F(x)-f(a,d'(\beta(e)))|<\inv{d'(\beta(e))}\\ \Rightarrow
    \Bigl|G\circ F(x)-g\Bigl(f\bigl(a, d'(\beta(e))\bigr), \beta(e)\Bigr)\Bigr|<\inv{\beta(e)}\leq \inv{e}
  \end{gather*}
  This shows also the second part of the theorem, only replace $V_e$
  and $O_e$ by $V\restriction e$ and $O\restriction e$.
\end{proof}

\begin{definition}
  A function $F:O\to U$ is called \F-compact if there is an
  \F--function $\beta:\N\to\N$ such that $F(O_e)\subseteq
  U_{\beta(e)}$ for all $e\in\Np$.
\end{definition}

Note that $F(x)=\inv{x}:(0,\infty)\to(0,\infty)$ and
$\ln(x):(0,\infty)\to(-\infty,\infty)$ are \low compact.

\begin{folgerung}\label{f:comp-c} Let $O$, $U$ and $V$ be open sets in
Euclidean space.
   If $F:O\to U$ is in \F and \F--compact, and if $G:U\to V$ is in \F, then
   $G\circ F$ is in \F. If $G$ is \F--compact, then so is $G\circ
   F:O\to V$.
\end{folgerung}
\begin{proof}
  By the proof of Lemma \ref{l:comp-u}.
\end{proof}

If $F$ is defined on the union of two open sets $U$ and $V$, and if
$F$ is in \F restricted to $U$ and restricted to $V$, it is not clear
that $F$ is in \F, without additional assumptions.
\begin{bemerkung}\label{b:union}
  Let $F$ be defined on the union of $U$ and $V$ and assume that
  $F\restriction U$ and $F\restriction V$ are in \F. Assume also
  \begin{enumerate}
  \item   that for some \F--function $u:\N\to\N$ $(U\cup V)_e\subset
  U_{u(e)}\cup V_{u(e)}$
  \item that $U$ and $V$ are \F--\ap in the sense of
    Definition~\ref{d:fopen} below,
  \end{enumerate}
  then $F$ is in \F.\qed
\end{bemerkung}
The two conditions are satisfied if $U$ and $V$ are open intervals
with \F--computable endpoints.
\section{Semialgebraic functions}

In this article \sa functions (relations) are functions (relations)
definable \emph{without parameters} in $\R$. The \emph{trace} of a
relation $R\subseteq \R^N$ on $\Q$ is $R\cap \Q^N$.

The following observation is due to Yoshinaga \cite{yoshinaga_periods}.
\begin{lemma}\label{l:spur}
  The trace of \sa relations on $\Q$ is \low.
\end{lemma}
\begin{proof}
  By quantifier elimination.
\end{proof}

Note that any \sa function $g:\R\to\R$ is polynomially bounded, i.e.
there is some $n\in\N$ with $|g(x)|\leq |x|^n$ for sufficiently large
$x$.
\begin{satz} \label{s:sa}
  Continuous \sa functions $F:O\to\R$ are \low for every open \sa set
  $O$.
\end{satz}
\begin{proof}
  Fix some large $e\in\Rp$. Note that the $O_e$ -- for real positive
  $e$ -- are uniformly definable. Since $F$ is uniformly continuous on
  $O_{2e}$, there is some positive real $d$ such that for all $x,x'\in
  O_{2e}$
  \[|x-x'|<\inv{d}\;\Rightarrow\;|F(x)-F(x')|<\inv{2e}.\]
  Since the infimum of all such $d$ is a \sa function of $e$, there is
  some $n\in\N$ such that for all large $e\in\N$ and all $x,x'\in O_{2e}$
  we have
  \[|x-x'|<\inv{e^n}\;\Rightarrow\;|F(x)-F(x')|<\inv{2e}.\]

  By a similar argument we obtain a polynomial bound $e^m$ for $|F|$
  on $O_{2e}$.

  Now define $f:\Q^N\times N\to\Q$
  in the following way: If $a$ does not belong to $O_{2e}$, set
  $f(a,e)=0$. Otherwise let
  $f(a,e)$ be the unique
  \[b\in\{-e^m,-e^m+\inv{2e},\ldots,e^m-\inv{2e},e^m\}\]
  such that $F(a)\in[b,b+\inv{2e})$.  Then $f$ is \low by
    Lemma~\ref{l:spur}.  Now assume $e\in\Np, a\in\Q^N$ and $x\in O_e$
    with $|x-a|<\inv{e^n}$. We may assume $2e\leq e^n$. Then $a\in
    O_{2e}$ and therefore
    \[|F(x)-f(a,e)|\leq|F(x)-F(a)|+|F(a)-f(a,e)|<\inv{2e}+\inv{2e}=
    \inv{e}.\]
\end{proof}

By Remark \ref{b:extends-uniform} this yields:
 \begin{folgerung}
 Let $F:O\to\R$ be \sa. If there is some open \sa set $U$ containing
 an $\epsilon$-neighborhood of $O$ and such that $F$ can be extended
 continuously and semialgebraically to $U$, then $F:O\to\R$ is
 uniformly \low.\qed
 \end{folgerung}

\begin{bemerkung}\label{b:sa_compact}
  It is easy to see that, if $F:O\to V$ is continuous and $F,O,V$ are
  \sa, then $F$ is \low compact.
\end{bemerkung}

 \begin{folgerung}\label{f:real_closed}
 The set of \F-reals $\R_\F$ forms a real closed field.
 \end{folgerung}

 \begin{proof}
 By Theorem~\ref{s:sa}, $\R_\F$ is a field.  To see that $\R_\F$ is
 real closed consider for odd $n\in\N$ the \sa function $f:\R^n\to \R$
 where $f(a_0,\ldots a_{n-1})$ is the minimal zero of the polynomial
 $\sum_{i=0}^{n-1} a_iX^i + X^n$.  By \sa cell decomposition (see
 \cite{vdDries}, Ch.~3), $\R^n$ can be decomposed into finitely many
 \sa cells on which $f$ is continuous. Each cell is homeomorphic to an
 open subset of $\R^k$ for some $k\geq 0$ via the appropriate (\sa)
 projection map. Thus, composing the inverse of such a projection
 $\pi$ with $f$ we obtain a \sa map on an open subset of $\R^k$.
 Applying Theorem~\ref{s:sa} and Lemma~\ref{l:constant} first to
 $\pi^{-1}$ and then to the composition, we see that a polynomial with
 coefficients in $\R_\F$ has a zero in $\R_\F$.
 \end{proof}

Corollary~\ref{f:real_closed} was first proved by Skordev
(see~\cite{skordev2}). For countable good classes \F, like the class
of \low functions, clearly $\R_\F$ is a countable subfield of $\R$.

\section{Integration}
\begin{satz}\label{s:int}
  Let $O\subset\R^N$ be open and $G,H:O\to\R$ in \F
  such that $G<H$ on $O$. Put
  \[U=\bigl\{(x,y)\in\R^{N+1}\mid x\in O,\; G(x)< y < H(x)\bigr\}.\]
  Assume further that $F:U\to\R$ is in \F and that $|F(x,y)|$ is bounded
  by an \F--function $K(x)$. Then
  \begin{equation}\label{e:int}
    I(x)=\int_{G(x)}^{H(x)}F(x,y)\dd y
  \end{equation}
  is an\/ \F--function $O\to\R$.
  \end{satz}
\begin{proof}
 We fix witnesses for $F,G,H$ being in \F: let
 $d:\N\to\N$, $f:\Q^{N+1}\times\N\to\Q$, $g,h:\Q^N\times \N\to\Q$ be
 \F--functions such that
 for $x\in O_e$, $a\in\Q^N$, $b\in\Q$
 \begin{align}
   |x-a|<\inv{d(e)}\Rightarrow\;&|G(x)-g(a,e)|<\inv{e} \mathrm{ \hspace{.5cm} and}\\
   & |H(x)-h(a,e)|<\inv{e}\\
   \intertext{and,
     if $|y|\leq e$ and $y\in[G(x)+\inv{e},H(x)-\inv{e}]$,}
   \label{e:f_absch} |x-a|<\inv{d(e)}&\land |y-b|<\inv{d(e)}\Rightarrow
   |F(x,y)-f(a,b,e)|<\inv{e}.
 \end{align}
 By Lemma \ref{l:bounded} we get an \F--function $\kappa:\N\to
 \N$ with $|F(x,y)|\leq\kappa(e)$  for all $(x,y)\in U$ and $H(x)-G(x)\leq\kappa(e)$
 for all $x\in O_e$. We assume that $\kappa\geq 2$.

 Let $d'(e)=d(e'')$, where
 $e'=12e\kappa(e)$ and
 $e''=2d(e')$. Define the function $j:\Q^{N}\times\N\to\Q$ as follows:
 if $h(a,e'')-g(a,e'')< \frac{4}{e'}$, set $j(a,e)=0$. Otherwise set
 \[j(a,e)=\sum_{s=0}^{S-1}f(a,g(a,e'')+\inv{e'}+s\delta,e')\cdot\delta,\]
where $S=\kappa(e)e''$ and
$\delta=\inv{S}(h(a,e'')-g(a,e'')-\frac{2}{e'})$. By Lemma \ref{l:q-add}
there is a function $i'(a,S,2e)=i(a,e)$ in \F with $|j(a,e)-i(a,e)|<\inv{2e}$.\\

In order to show that for all $a\in\Q^N$ and $x\in O_e$
\begin{equation}
 |x-a|<\inv{d'(e)} \Rightarrow  \bigl|I(x)-i(a,e)\bigr| < \inv{e}\label{e:i}
\end{equation}
 it suffices to show
  \[|x-a|<\inv{d'(e)} \Rightarrow  \bigl|I(x)-j(a,e)\bigr| < \inv{2e}.\]

Since\footnote{Remember that $d$ is strictly increasing by
  assumption.} $e\leq e'\leq e''$, the hypothesis implies $x\in
O_{e'}\subseteq O_{e''}$.  We also have $|x-a|<\inv{d(e'')}$ and
therefore
\begin{align}
  |G(x)-g(a,e'')|&<\inv{e''}\label{e:g}\\
  |H(x)-h(a,e'')|&<\inv{e''}\label{e:h}.
\end{align}

\noindent First case: $h(a,e'')-g(a,e'')< \frac{4}{e'}$. Using
\eqref{e:g} and \eqref{e:h} we have
$|H(x)-G(x)|<\frac{2}{e''}+\frac{4}{e'}\leq\frac{6}{e'}$. Therefore
$\bigl|I(x)\bigr|<\frac{6}{e'}\kappa(e)=\inv{2e}$, which was to be
shown.\\

\noindent Second case: $h(a,e'')-g(a,e'')\geq\frac{4}{e'}$. This
implies $H(x)-G(x)\geq\frac{4}{e'}-\frac{2}{e''}\geq\frac{2}{e'}$.
Thus $\Delta=\inv{S}(H(x)-G(x)-\frac{2}{e'})$ is non-negative. Note
that by \eqref{e:g} and \eqref{e:h}
\begin{equation}
  |\Delta-\delta|<\frac{2}{Se''}.
\end{equation}
We also have \[\Delta<\frac{\kappa(e)}{S}=\inv{e''}.\]
It is easy to see that for each $s\leq S$, we have
\begin{multline}
\bigl|\bigl(G(x)+\inv{e'}+s\Delta\bigr)-\bigl(g(a,e'')+
\inv{e'}+s\delta\bigr)\bigr|\leq\\
\max\bigl(|G(x)-g(a,e'')|,|H(x)-h(a,e'')|\bigr)<\inv{e''}.
\end{multline}
This implies for every
$y\in[G(x)+\inv{e'}+s\Delta,G(x)+\inv{e'}+(s+1)\Delta]$ that
\[\bigl|y-\bigl(g(a,e'')+\inv{e'}+s\delta\bigr)\bigr|<\inv{e''}+
\Delta\leq\frac{2}{e''}=\inv{d(e')}.\]
Since $|x-a|<\inv{d(e')}$, we have therefore
\begin{equation}\label{e:F-f}
  \bigl|F(x,y)-f\bigl(a,g(a,e'')+\inv{e'}+s\delta,e'\bigr)\bigr|<\inv{e'},
\end{equation}
which implies
\begin{equation}\label{e:1}
  \Bigl|\;\int_{G(x)+\inv{e'}}^{H(x)-\inv{e'}}F(x,y)\dd y\;\;-\;\;
  \sum_{s=0}^{S-1}f\bigl(a,g(a,e'')+\inv{e'}+s\delta,e'\bigr)\cdot\Delta\;
  \Bigr|<\frac{\kappa(e)}{e'}.
\end{equation}
By \eqref{e:F-f} we have that all
$f\bigl(a,g(a,e'')+\inv{e'}+s\delta,e'\bigr)$ are bounded by
$\kappa(e)+1$. It follows that
\begin{multline}\label{e:2}
  \Bigl|\;
  \sum_{s=0}^{S-1}f\bigl(a,g(a,e'')+\inv{e'}+s\delta,e'\bigr)\cdot\Delta\;
  \;-\;j(a,e)\;\Bigr|\leq\\
  S(\kappa(e)+1)|\Delta-\delta|\leq\frac{2(\kappa(e)+1)}{e''}.
\end{multline}
Finally the absolute values of $\int_{G(x)}^{G(x)+\inv{e'}}F(x,y)\dd
y$ and $\int_{H(x)-\inv{e'}}^{H(x)}F(x,y)\dd y$ are bounded by
$\frac{\kappa(e)}{e'}$. By \eqref{e:1} and
\eqref{e:2}, this yields
\[\bigl|I(x)-j(a,e)\bigr|<\frac{2(\kappa(e)+1)}{e''}+
\frac{\kappa(e)}{e'}+\frac{2\kappa(e)}{e'}\leq
\frac{6\kappa(e)}{e'}=\inv{2e}.\]
This proves (\ref{e:i}).
\end{proof}

An examination of the proof yields:

\begin{folgerung}
  If $F,G,H$ are as above and uniformly in \F, then the function
  \[x\mapsto\int_{G(x)}^{H(x)}F(x,y)\dd y\]
  is uniformly in \F.
\end{folgerung}

\section{Periods}\label{sec:periods}
Kontsevich and Zagier \cite{kontsevich-zagier} define a period as
\emph{a complex number whose real and imaginary parts are values of
absolutely convergent integrals of rational functions with rational
coefficients over domains in $\R^n$ given by polynomial
inequalities with rational coefficients.} The periods form a ring
containing all algebraic reals. It is an open problem whether $e$ is a
period.

Yoshinaga \cite{yoshinaga_periods} proved that periods are
\emph{\el{}.} An analysis of his proof shows that he actually showed
that periods are \low. We here give a variant of his proof where part
of his argument is replaced by an application of Theorem \ref{s:int}.

\begin{definition}
  A $1$--dimensional bounded open cell is a bounded open interval with
  algebraic endpoints.  An $n+1$--dimensional bounded open cell is of
  the form
  \[\bigl\{(x,y)\in\R^{N+1}\mid x\in O,\; G(x)< y < H(x)\bigr\}\]
  for some $n$--dimensional bounded open cell $O$ and
  bounded continuous \sa functions $G<H$ from $O$ to $\R$.
\end{definition}
Thus our  cells are \sa.
\begin{lemma}
    Let $C$ be an $N$--dimensional bounded open cell and $N=A+B$. Let
    $O$ be the projection on the first $A$--coordinates and for $x\in
    O$ let $C_x$ be the fiber over $x$. Then the map
    $x\mapsto\vol(C_x)$ is a bounded \low function.
\end{lemma}
\begin{proof}
    By induction on $B$. Let $U$ be the projection on the first $A+1$
    coordinates. Then $U$ is of the form
    \[\bigl\{(x,y)\in\R^{A+1}\mid x\in O,\; G(x)< y <
    H(x)\bigr\}\] for some bounded and continuous \sa functions $G<H$.
    By induction hypothesis, $u\mapsto\vol(C_u)$ is a bounded \low
    function $U\to\R$. By Fubini
    \[\vol(C_x)=\int_{G(x)}^{H(x)}\vol(C_{x,y})\dd y.\] This is
    bounded and \low by Theorem \ref{s:int}.
\end{proof}
\begin{folgerung}
  The volumes of bounded \sa sets are \low.
\end{folgerung}
\begin{proof}
  Since every \sa set is the disjoint union of \sa cells, it is enough
  to know the claim for bounded \sa cells, which is the $N=0$ case of
  Lemma \ref{l:constant}.
\end{proof}

\begin{folgerung}
    Periods are \low.
\end{folgerung}
\begin{proof}
    By Lemma 24 of \cite{yoshinaga_periods}, periods are differences of
    sums of volumes of bounded open \sa cells.
\end{proof}
\section{The Inverse Function Theorem}\label{sec:ift}

We call a sequence $\A_1, \A_2,\ldots$ of subsets of $\Q^N$ an
\emph{\F--sequence}, if $\{(e,a)\mid a\in\A_e\}$ is an \F--subset of
$\Np\times \Q^N$.

\begin{definition}\label{d:fopen}
  An open set $O\subset\R^N$ is \F--\ap if there is an \F--sequence
  $\A_1, \A_2,\ldots$ of subsets of $\Q^N$ and an \F--function
  $\alpha:\N\to\N$ such that
  $O_e\cap\Q^N\subset\A_e\subset O_{\alpha(e)}$ for all $e\in\Np$.
\end{definition}
It follows from Lemma \ref{l:spur} that \sa sets $O$ are \low \ap. We
can simply set $\A_e=O_e\cap\Q^N$.

The following observations will not be used
in the sequel. For the last part, we need Lemma~\ref{l:gamma}.

\begin{bemerkung}\label{b:approx}
  \begin{enumerate}
  \item\label{b:approx:intervall}The intervals $(x,\infty)$,
    $(-\infty,y)$ and  $(x,y)$ are \F--\ap if
    and only if $x$ and $y$ are \F--reals.
  \item\label{b:approx:cell} If $O$ is \F--\ap and $F, G: O\to\R$ are
    in \F, then $\{(x,y)\mid x\in O, F(x)<y<G(x)\}$ is \F--\ap.
  \item\label{b:approx:homeo} If $F:O\to V$ is a homeomorphism and $F$
    and $F^{-1}$ are uniformly in \F, then $O$ is \F--\ap if and only if
    $V$ is.
  \end{enumerate}
\end{bemerkung}

\begin{satz}\label{s:ift}
  Let $F:O\to V$ be a bijection in \F where $O$ is \F--\ap and $V$
  open in $\R^N$. Assume that the inverse $G:V\to O$ satisfies:
  \begin{enumerate}[(i)]
  \item\label{s:ift:gleichm_stetig} There is an \F--function
    $d':\N\to\N$ such that $|G(y)-G(y')|<\inv{e}$ for all $y,y'\in
    V_e$ with $|y-y'|<\inv{d'(e)}$.
  \item\label{s:ift:compact} $G$ is \F--compact.
  \end{enumerate}
  Then $G$ is also in \F.
\end{satz}
By the proof of Lemma \ref{l:stetig} and Remark \ref{b:gamma} below
the conditions (\ref{s:ift:gleichm_stetig}) and (\ref{s:ift:compact})
are necessary for the conclusion to hold.
\begin{proof}
  As $G$ is \F--compact, let $\gamma:\N\to \N$ be an \F--function such
  that $G(V_e)\subseteq O_{\gamma(e)}$.

  Since $F$ is in \F, we find \F--functions $d(e)$ and $f(a,e)$ such
  that for all $a\in\Q^N$ and $x\in O_e$
\begin{equation}\label{e:elu}
  |x-a|<\inv{d(e)} \to |F(x)-f(a,e)| < \inv{e}
\end{equation}
We also fix a function $\alpha$ and a sequence $\A_i,i\in\N_{>0}$ as in
Definition \ref{d:fopen}.

We now construct two \F--functions $d'':\N\to\N$ and
$g:\Q^N\times\N\to\Q^N$ such that
\begin{equation}\label{e:g-stetig}
  |y-b|<\inv{d''(e)} \to |G(y)-g(b,e)| < \inv{e}
\end{equation}
for all $e\in\Np$, $b\in\Q^N$ and $y\in V_e$.

Fix $e\in\N$ and $b\in\Q^N$. Set
\[d''=d''(e)=\max\bigl(\,4d'(2e),\,8e,
\,\alpha(2\gamma(e)), \,2\gamma(e)\bigr).\]
Also put $C=\max\bigl(\,2\gamma(e),\,d(d'')\bigr)$ and consider
the set
\[\A=\A_{2\gamma(e)}\cap\bigl(\,\inv{C}\,\,\Z\bigr)^N.\]
Since the elements of $\A$ are bounded by $\alpha(2\gamma(e))$, $\A$
is a finite set. If there is an $a\in\A$ such that
\begin{equation}\label{e:fab}
  \bigl|b-f(a,d'')\bigr|<\frac{2}{d''},
\end{equation}
we choose such an $a$ by an \F--function (!) $a=g(b,e)$. Otherwise put $g(b,e)=0$.\\

Let us check that $d''$ and $g$ satisfy \eqref{e:g-stetig}. Start with
$e$ and $b$ as above and consider an $y\in V_e$ with
$|y-b|<\inv{d''(e)}$.

We first show that $\A$ contains an element $a'$ with
$|b-f(a',d'')|<\frac{2}{d''}$.  For this set $x=G(y)\in O_{\gamma(e)}$.
Choose an
$a'\in\bigl(\inv{C}\,\,\Z\bigr)^N$ such that $|x-a'|<\inv{C}$. Since
$C\geq 2\gamma(e)$, we have $a'\in O_{2\gamma(e)}$ and therefore
$a'\in\A$. Since $d''\geq2\gamma(e)$ and $d(d'')\leq C$ we have $|y-f(a',d'')|<\inv{d''}$ by
\eqref{e:elu}. This implies
$|b-f(a',d'')|<\frac{2}{d''}$.

Now set $a=g(b,e)$. By the previous paragraph, we know that $a$ is in
$\A$ and satisfies \eqref{e:fab}. Since $a\in O_{\alpha(2\gamma(e))}$
and $d''\geq\alpha(2\gamma(e))$, we have
$\bigl|F(a)-f(a,d'')\bigr|<\inv{d''}$.  This implies
$|y-F(a)|\leq|y-b|+|b-f(a,d'')|+|f(a,d'')-F(a)|<\frac{4}{d''}$.  Since
$\frac{d''}{4}\geq 2e$, this implies $F(a)\in V_{2e}$.
Since $d'(2e)\leq\frac{d''}{4}$, we can use (\ref{s:ift:gleichm_stetig})
to obtain $|x-a|<\inv{e}$.
\end{proof}

\begin{lemma}\label{l:gamma}Let $G:V\to O$ be open and continuous.
  Suppose
  \begin{enumerate}
  \item\label{l:gamma:schranke} $G$ is \F--bounded.
  \item\label{l:gamma:gleichm_stetig} There is an \F--function
    $d':\N\to\N$ such that
    \[|G(x)-G(x')|<\inv{d'(e)}\;\to\;|x-x'|<\inv{e}\]
    whenever $G(x),G(x')\in O\restriction e$
  \end{enumerate}
  Then $G$ is \F--compact.
\end{lemma}
\begin{bemerkung}\label{b:gamma}
  By Lemma \ref{l:bounded} and the proof of Lemma \ref{l:stetig},
  respectively, the conditions of Lemma~\ref{l:gamma} are satisfied if $G$ is a
  bijection in \F and $G^{-1}$ is uniformly in \F.
\end{bemerkung}
\begin{proof}[Proof of \ref{l:gamma}]
  Choose an \F--function $\beta$ such that $|G(x)|\leq \beta(e)$ for
  all $x\in V_e$, so $G(V_e)\subseteq O\restriction \beta(e)$.  We may
  assume $\beta(e)\geq 2e$.  Let $\gamma(e)=\max(2\beta(e),
  d'(2\beta(e)))$. We will show that $G(V_e)\subset O_{\gamma(e)}$ for
  all $e\in\Np$.

  Let $x\in V_e$ and $y=G(x)\in O\restriction
  \beta(e)\subseteq O\restriction \gamma(e)$. So we have to show that
  $\di(y,\R^N\setminus O)\geq\inv{\gamma(e)}$.  Let $B$ be the open
  ball\footnote{Actually this is a cube, since we use the maximum
    norm. } around $x$ with radius $\inv{e+1}$. Then the closure
  $\overline B=B\cup\delta B$ is still a subset of $V$. Since $G$ is
  continuous and open, $G(\overline B)$ is compact and $G(B)$ is open
  in $O$ and therefore in $\R^N$. Let $y''$ be any element in
  $\R^N\setminus O$. Look at the line segment $L$ between $y''$ and
  $y$. $L$ contains an element of $G(\delta B)$ since otherwise the
  traces of $G(B)$ and $\R^N\setminus G(\overline B)$ on $L$ were
  an open partition of $L$. So let $x'\in\delta B$ and $y'=G(x')\in
  L$. Then $|x-x'|= \inv{e+1}\geq \frac{1}{2e}$.

  By assumption
  \[|y-y'|<\frac{1}{d'(2\beta(e))} \to |x-x'| < \frac{1}{2\beta(e)}\]
  if $y,y'\in O\restriction 2\beta(e)$.  So there are two
  cases. Either $|y-y'|\geq
  \frac{1}{d'(2\beta(e))}\geq\inv{\gamma(e)}$ or $|y'| > 2\beta(e)$.
  But since $|y|\leq\beta(e)$, in this case we have $|y-y'| > 1$ and we
  are done either way.
\end{proof}

\begin{folgerung}\label{f:ift}
  Let $O$ be \F--\ap and\/ $V$ open and convex in $\R^N$ and $F:O\to V$
  a bijection which is uniformly in \F.  Assume that the inverse $G:V\to
  O$ is differentiable and that $|\D(G)|$ can be bounded by an
  \F--function $G':V\to\R$. Then $G$ belongs also to \F.
\end{folgerung}
\noindent Here the norm of a matrix $A=(a_{i,j})$ is
$\max_{i}\sum_j|a_{i,j}|$.
\begin{proof}
It suffices to show that $G$ satisfies (\ref{s:ift:gleichm_stetig})
and (\ref{s:ift:compact}) of Theorem~\ref{s:ift}.

\noindent Proof of (\ref{s:ift:gleichm_stetig}): By Lemma
\ref{l:bounded}, there is an \F-function $\gamma:\N\to\Q$ such that
$|G'(y)|\leq\gamma(e)$ for all $y\in V_e$. Assume that $y,y'\in V_e$
and $|y-y'|<\inv{2e}$. Then the line segment between $y$ and $y'$ is
contained in $V_{2e}$ and it follows that
$|G(y)-G(y')|\leq|y-y'|\gamma(2e)$
So we can set $d'(e)=\max\bigl(2e,\frac{e}{\gamma(2e)}\bigr)$.\\

\noindent Proof of (\ref{s:ift:compact}): We have to verify the two
conditions of Lemma \ref{l:gamma}. Condition \ref{l:gamma:gleichm_stetig}
follows from the assumption that $F$ is uniformly in \F. It remains to show
that $G$ is \F-bounded.

 Fix some $y_0\in V$ and some $e_0$ with $y_0\in V_{e_0}$. If $e\geq
 e_0$ and $y\in V_e$, then the line segment between $y_0$ and $y$ lies
 in $V_e$. So
 \[|G(y)|\leq|G(y)-G(y_0)|+|G(y_0)|\leq|y-y_0|\gamma(e)+|G(y_0)|.\]
 We set $\beta(e)=2e'\gamma(e')+\lceil|G(y_0)|\rceil$, where
 $e'=\max(e,e_0)$, and have $|G(y)|\leq\beta(e)$.
\end{proof}
 The next proposition shows that for proper intervals we can weaken the assumptions:
\begin{proposition}\label{p:ifti}
  Let $F:O\to V\subseteq \R$ be a homeomorphism
  with inverse $G$ where $O=(c_0,c_1)$ is a bounded open interval
  whose endpoints are \F--reals.
  Suppose that $F$ belongs to \F  and that there is an
  \F--function $d':\N\to\N$ such that $|G(y)-G(y')|<\inv{e}$ for all
  $y,y'\in V_e$ with $|y-y'|<\inv{d'(e)}$. Then $G$ is in \F.
\end{proposition}
\begin{proof} Without loss of generality let us assume that $F$ is increasing.
 For simplicity we also assume $c_0,c_1\in\Q$.
   Since $F$ is in \F, we find \F--functions $d(e)$ and $f(a,e)$ such that
  for all $a\in\Q$ and $x\in O_e$ we have
  \[
  |x-a|<\inv{d(e)} \to |F(x)-f(a,e)| < \inv{e}
  \]
  We may assume $d(e)\geq 4e$ and $d'(e)\geq e$ for all $e\in\N$.
  We will find an \F--function $g(b,e)$ such that
  for  $y\in V_e$ we have
  \[
  |y-b|<\inv{d'(2d(e))}\to |G(y)-g(b,e)|<\inv{e}
  \]

  Let $e'=2d'(2d(e))$.  If $\inv{2d(e')}\Z\cap O_{e'}=\emptyset$, put
  $g(b,e)=\inv{2}(c_0+c_1)$.  If there is some $a\in\inv{2d(e')}\Z\cap
  O_{e'}$ with $|b-f(a,e')|<\inv{e'}$, put $g(b,e)=a$ with $a$ minimal
  such.  Otherwise, put $g(b,e)=c_1$ if $b-f(a,e')\geq\inv{e'}$ for
  all $a\in\inv{2d(e')}\Z\cap O_{e'}$ and $c_0$ if
  $f(a,e')-b\geq\inv{e'}$ for all $a\in\inv{2d(e')}\Z\cap O_{e'}$.
  Note that one of these cases occurs since for
  $a,a'\in\inv{2d(e')}\Z\cap O_{e'}$ with $|a-a'|=\inv{2d(e')}$ we
  have $|f(a,e')-f(a',e')|<\frac{2}{e'}$.\\

  \noindent
  Now let $y\in V_e$ and $|y-b|<\inv{d'(2d(e))}$.  So $b,y\in
  V_{2d(e)}$ and $ |G(y)-G(b)|<\inv{2d(e)}$.

  \noindent {\bf Case I:} Suppose that $x=G(y)\in O_e$.
   Note that $O_e\neq\emptyset$ implies that
   $\inv{2d(e')}\Z\cap O_{e'}\neq\emptyset$.
   Since
  $|G(y)-G(b)|<\inv{2d(e)}$, we have $G(b)\in O_{e'}$. Hence there is
  some smallest $a\in\inv{2d(e')}\Z\cap O_{e'}$ with
  $|b-f(a,e')|<\inv{e'}$.
  Then $g(b,e)=a$ and
  \[|G(y)-g(b,e)|\leq |G(y)-G(b)|+|G(b)-a|.\]

  Now
  \[|F(a)-b|\leq  |F(a)-f(a,e')|+|f(a,e')-b|\leq 2\inv{e'}\leq
  \inv{d'(2d(e))}.\] Since $y\in V_e$ and $|y-F(a)|\leq |y-b|+|F(a)-b|<
\frac{2}{d'(2d(e))}$
  we have $F(a),b\in V_{2d(e)}$ and hence
  $|G(b)-a|\leq \inv{2d(e)}$. Combining all this we see
  that
  \[|G(y)-g(b,e)|<\inv{d(e)}<\inv{e}.\]

  \noindent {\bf Case II:} Suppose that $x=G(y)\notin O_e$. Then if
  $g(b,e)=a\in O_{e'}$
  the same argument as above works. Otherwise by construction of $g$,
  $g(b,e)$ equals either
  $c_0$ or $c_1$ and  we again have
  $|G(y)-g(b,e)|<\inv{e}$.
\end{proof}

\section{Series of functions}
\begin{definition}\label{d:f-sequence}
  For an open set $O\subseteq R^N$ a sequence $F_1,F_2,\ldots$ of functions $O\to\R^M$ is in \F if there
  are \F--functions $d:\N^2\to\N$ and $f:\Q^N\times\N^2\to\Q^M$ such
  that for all $i,e\in\Np$ and all $a\in\Q^N$ and $x\in O_e$
  \begin{equation}\label{e:f-sequence}
    |x-a|<\inv{d(i,e)} \to |F_i(x)-f(a,i,e)| < \inv{e}.
  \end{equation}
\end{definition}

This definition is the $n=1$ case of the obvious notion of an
\F--function $F:\N^n\times O\to\R^M$. Note also that for $N=0$ this
defines \F--sequences of elements of $\R^M$.

\begin{lemma}\label{l:f-function-sequence}
  If $F:O\times\Rp\to\R^M$ is in \F, then so is the sequence
  $F(-,i):O\to\R^M$, $(i=1,2,\ldots)$.
\end{lemma}
\begin{proof}
  There are \F--functions $d':\N\to\N$ and $f:\Q^{N+1}\times\N\to\Q^M$
  such that for all $a\in\Q^N$, $x\in O_e$ and $b\in[\inv{e},e]\cap\Q$
  \[
  |x-a|<\inv{d'(e)} \to |F(x,b)-f(a,b,e)| < \inv{e}.
  \]
  So we can set $d(i,e)=d'(\max(i,e))$.
\end{proof}
\begin{definition}\label{d:f-convergence}
  A sequence $F_1,F_2,\ldots$ of functions $O\to\R^M$\,\, \F--converges
  against $F$, if there is an \F--function $m:\N\to\N$ such that
  $|F(x)-F_i(x)|<\inv{e}$ for all $x\in O_e$ and $i\geq m(e)$.
\end{definition}
\begin{lemma}
  The \F--limit of an \F--sequence of functions is an \F--function.
\end{lemma}
\begin{proof}
  Let $d$ and $f$ be as in Definition \ref{d:f-sequence} and $m$ as in
  \ref{d:f-convergence}. Set $d'(e)=d(m(2e),2e)$ and
  $f'(a,e)=f(a,m(2e),2e)$. Consider $a\in\Q^N$, $x\in O_e$ and assume
  $|x-a|<\inv{d'(e)}$. Then $|F(x)-F_{m(2e)}(x)|<\inv{2e}$ and
  $|F_{m(2e)}(x)-f'(a,e)|<\inv{2e}$. It follows that
  $|F(x)-f'(a,e)|<\inv{e}$.
\end{proof}
\begin{proposition}\label{p:f-series}
  Let $F_1,F_2,\ldots$ be an \F--sequence of functions $O\to\R^M$
  such that the partial sums of the series $\sum_{i=1}^\infty F_i$ are
  \F--convergent. Then $\sum_{i=1}^\infty F_i:O\to\R^M$ is in \F.
\end{proposition}
\begin{proof}
  We have to show that the series of partial sums is an \F--series of
  functions. This can easily be done using Lemma \ref{l:q-add}.
\end{proof}
\section{Examples}
\newcounter{bspl}\setcounter{bspl}{0}
\nc{\beispiel}[1]{\refstepcounter{bspl}
\noindent{\bf\arabic{bspl}. #1}\\}
\newcounter{fact}\setcounter{fact}{0}
\nc{\factum}{\refstepcounter{fact}
\noindent{\bf Fact \arabic{bspl}.\arabic{fact}: }}

\nc{\re}{\mathrm{Re}}
\nc{\im}{\mathrm{Im}}
\nc{\Cs}{\C\setminus\R_{\scriptscriptstyle\leq0}}

\beispiel{Inverse trigonometric functions}\\ The function
$\frac{x}{1+x^2t^2}$ is continuous and \sa. So by Theorem \ref{s:int}
\[\arctan(x)=\int_0^1\frac{x}{1+x^2t^2}\dd t\] is \low.
The same argument shows that
\[\arcsin(x)=\int_0^1\frac{x}{\sqrt{1-x^2t^2}}\dd t\] is \low\footnote{
We do not know whether $\arcsin$ is uniformly \low.}, as a function
defined on $(-1,1)$.\\

\beispiel{Logarithm}\\ As a \sa function $\inv{x}:\Rp\to\R$ is
\low. So by Theorem \ref{s:int} $\ln(x)=\int_1^x\inv{y}\dd y$ is \low,
at least on $(0,1)$ and $(1,\infty)$. But writing
\[\ln(x)=\int_0^1\frac{x-1}{1+t(x-1)}\dd t,\]
we see that $\ln(x)$ is \low on $(0,\infty)$.

The same formula defines also the main branch of the complex logarithm
$\ln(z):\Cs\to\C$. Since the real and imaginary part of the integrand
$\frac{z-1}{1+t(z-1)}$ are \sa functions of $\re(z)$, $\im(z)$ and
$t$, we conclude that $\ln(z)$ is \low.\footnote{We identify $\C$ with
  $\R^2$.} It is also easy to see that $\ln(z)$ is \low compact as a
function from $\Cs$ to $\{z\mid \im(z)\in(-\pi,\pi)\,\}$.\\

\beispiel{Exponentiation} \\ As $\exp(x)$ is bounded on every interval
$(-\infty, r)$ we may apply Proposition \ref{p:ifti} to
$\ln:(0,1)\to(-\infty,0)$ to conclude that $\exp(x)$ is \low on
$(-\infty,0)$ and -- by translation -- on every interval $(-\infty,
r)$. $\exp(x)$ cannot be \low on the whole real line since it grows
too fast. Nevertheless the following is true.

\begin{lemma}\label{l:bounded_exp}
$G(x,y)=\exp(x)$ is \low on $V=\{(x,y)\colon\exp(x)<y\}$.
\end{lemma}
\begin{proof}
Clearly, $G$ is \low on $V\cap(\R_{<1}\times \R)$ and differentiable
on $V$.  So let us consider $V'=V\cap (\R_{>0}\times \R)$, which is a
convex open subset of $\R^2$.  Let $O=\{(z,y)\colon 1<z<y\}$, which is
\low \ap, let $F: O\longrightarrow V'$ map $(z,y)$ to $(\ln
z,y)$ and let $H$ be the inverse of $F$. The norm of the differential
of $H$ is bounded by the \low function $(x,y)\mapsto y$.  Now $F$ is
uniformly low, which by Corollary~\ref{f:ift} implies that $H$ and
therefore $G\upharpoonright V'$ is \low. Now $G$ is \low on $V$ by
Remark \ref{b:union}.
\end{proof}

The complex logarithm defines a homeomorphism \[\ln:\Cs\to\{z\mid
\im(z)\in(-\pi,\pi)\,\}\] and for all $r<s\in\R$ a uniformly \low
homeomorphism between $\{z\in\Cs\mid\exp(r)<|z|<\exp(s)\}$ and
$\{z\mid \im(z)\in(-\pi,\pi), r<\re(z)<s\}$. We can apply Corollary
\ref{f:ift} to see that $\exp(z)$ is \low on $\{z\mid
\im(z)\in(-\pi,\pi), r<\re(z)<s\}$. Using the periodicity of $\exp(z)$
it is now easy to see that $\exp(z)$ is \low on each strip $\{z\mid
r<\re(z)<s\}$. This implies that $\sin(x):\R\to\R$ is \low. Now also
$\cos(x)$ is \low and since $\exp(x+y\ii)=\exp(x)(\cos(y)+\sin(y)\ii)$
we see that $\exp(z)$ is \low on every half-space
$\{z\mid\re(z)<s\}$.\\

\begin{bemerkung}
  It is easy to see that $\exp(z)$ is elementary on $\C$.
\end{bemerkung}\vspace{2em}

\beispiel{$x^y$}\label{bspl:xy} \\
Consider the function $x^y$, defined on
$X=(\Cs)\times\C$ by
\[x^y=\exp(\ln(x)\cdot y).\]
Since $\ln(x)$ is \low and $\exp(x)$ elementary on $\C$, it is clear
that $x^y$ is elementary.\\

Let us determine some subsets of $X$ on which $x^y$ is \low. We use
the notation $E=\{z\in\C\mid|z|<1\}$ for the open unit disc.\\

\factum $x^y$ is \low on $(\Es)\times\Rp$.
\label{fact:unit-positiv-reell}\\

\noindent Proof: $\ln(x)\cdot y$ maps this set to $\{z\mid\re(z)<0\}$,
on which $\exp$ is uniformly \low.\\

\factum $x^y$ is \low on
$(0,1)\times\{z\mid\re(z)> 0\}$.\label{fact:eins-positiv}\\

\noindent Proof: $\ln(x)\cdot y$ maps this set to
$\{z\mid\re(z)<0\}$.\\

\factum $x^y$ is \low on $\Rp\times\{z\mid 0<\re(z)<r)$, for all
positive $r$.\label{fact:positiv-beschraenkt}\\

\noindent Proof: Since
$x^{a+b\ii}=x^a(\sin(\ln(x)b)+\cos(\ln(x)b)\ii)$, it suffices to
consider $x^y$ on $\Rp\times(0,r)$. $x^y$ is \low on
$(0,2)\times(0,r)$, since $\ln(x)y$ maps $(0,2)\times(0,r)$ to
$(-\infty,\ln(2)r)$. Therefore we are left with
$U=(1,\infty)\times(0,r)$. Let $N$ be a natural number $\geq
r+1$. Then $x^y$ is the composition of $F(x,y)=(\ln(x)y,x^N)$ and
$G(z,w)=\exp(z)$. $F$ maps $U$ into $V=\{(z,w)\mid \exp(z)<w\}$, on
which $G$ is \low by Lemma \ref{l:bounded_exp}. We will show that
$F:U\to V$ is \low compact, so we can apply Corollary~\ref{f:comp-c}
to conclude that $x^y$ is \low on $U$. It is clear that $F$ is \low
compact as a function from $U$ to $\R^2$. So it suffices to show that
for all $e\in\Np$ and $(x,y)\in[\inv{e},\infty)\times(0,r)$ we have
  $\di(F(x,y),\R^2\setminus V)\geq\inv{4e}$. This amounts to
  $\exp(\ln(x)y+\inv{4e})<x^N-\inv{4e}$, which is easy to prove:
  {\renewcommand{\inv}[1]{\tfrac{1}{#1}}
  \begin{align*}
    \exp\bigl(\ln(x)y+\inv{4e}\bigr)&=x^y\exp\bigl(\inv{4e}\bigr)\leq
    x^y\bigl(1+\inv{2e}\bigr)<
    x^y\bigl(1+\inv{e}\bigr)-x^y\inv{4e}\\
    &\leq x^N-x^y\inv{4e}\leq x^N-\inv{4e}.
  \end{align*}
}

\beispiel{Gamma function}

We have for $\re(x)>1$
\begin{align*}
\Gamma(x)&=\int_0^\infty t^{-1+x}\exp(-t)\dd t\\
&=
\int_0^1 t^{-1+x}\exp(-t)\dd t+\int_1^\infty t^{-1+x}\exp(-t)\dd t\\
&=
\int_0^1 t^{-1+x}\exp(-t)\dd
t+\int_0^1\inv{t^{1+x}}\exp\bigl(\frac{-1}{t}\bigr)\dd t.
\end{align*}
Let us check that for every bound $r>1$ the two integrands are \low on
$X_r=\{(x,t)\mid\re(x)\in(1,r), t\in(0,1)\}$: $\exp(-t)$ is \low on
$(0,1)$. And, since $\inv{t}:(0,1)\to(1,\infty)$ is \low compact
(cf.\ Remark \ref{b:sa_compact}), the function $\exp(\frac{-1}{t})$ is
\low on $(0,1)$.  $t^{-1+x}$ and
$\inv{t^{1+x}}=\big(\inv{t}\bigr)^{x+1}$ are \low by Fact
\ref{bspl:xy}.\ref{fact:positiv-beschraenkt} above.

If $r>1$, the absolute values of the integrands are bounded by $1$ and
$(r+1)^{(r+1)}$, respectively. So by Theorem \ref{s:int}, $\Gamma$ is
\low on every strip $\{z\mid1<\re(z)<r\}$. \\

\beispiel{Zeta--function}\\ Since $x^y$ is \low on
$(0,1)\times\{z\mid\re(z)>0\}$ (Fact
\ref{bspl:xy}.\ref{fact:eins-positiv}), the function $(\inv{x})^y$ is
\low on $(1,\infty)\times\{z\mid\re(z)>0\}$. This implies that the
sequence $\inv{n^s}$, $(n=1,2,\ldots)$ is a \low sequence of functions
defined on $\{z\mid\re(z)>0\}$ by Lemma
\ref{l:f-function-sequence}\footnote{Strictly speaking, one applies
  \ref{l:f-function-sequence} to $(\inv{x+1})^y$ to get the \low
  series $\inv{2^s},\inv{3^s},\ldots$.}. The series
\[\zeta(z)=\sum_{n=1}^{\infty}\inv{n^z}\]
converges whenever $t=\re(z)>1$ and we have the estimate
\[\Bigl|\zeta(z)-\sum_{n=1}^{N}\inv{n^z}\Bigr|\;\leq\;
\int_N^\infty\inv{x^t}\dd x= \inv{(t-1)N^{t-1}}.\] So, if $\re(z)\geq
1+\inv{k}$ and $N\geq(ke)^k$, we have
$\Bigl|\zeta(z)-\sum_{n=1}^{N}\inv{n^z}\Bigr|<\inv{e}$.  This shows
that $\zeta(z)$ is \low on every $\{z\mid\re(z)>s\}$, $(s>1)$ by
Proposition \ref{p:f-series}.

\section{Holomorphic functions}
\begin{lemma}
  The sequence of functions $z^n$, $n=0,1,\ldots$ is a \low sequence
  of functions on $\E$
\end{lemma}
\begin{proof}
  By Fact \ref{bspl:xy}.\ref{fact:unit-positiv-reell} and Lemma
  \ref{l:f-function-sequence}, the sequence $z^1$, $z^2$,\ldots\ is
  \low on $\Es$. Since $(-z)^n=(-1)^nz^n$, it follows that $z^1$,
  $z^2$,\ldots\ is \low on $\E\setminus\{0\}$. It is now easy to see
  that $z^1$, $z^2$,\ldots\ is actually \low on $\E$. (Set
  $f(a,i,e)=0$, if $|a|<\inv{e}$.)
\end{proof}
\begin{lemma}\label{l:sequence-low}
  Let $F(z)=\sum_{i=0}^\infty a_iz^i$ be a complex power series with
  radius of convergence $\rho$. Let $0<b<\rho$ be an \F--real such
  that $(a_ib^i)_{i\in\N}$ is an \F--sequence of complex
  numbers. Then $F$ restricted to the open disc $\{z\colon |z|<b\}$ belongs
  to \F.
\end{lemma}
\begin{proof}
  By assumption and the last lemma the sequence $(a_ib^iz^i)$ is \F on
  $\E$. If we plug in the \F--function $z\mapsto\frac{z}{b}$, we see
  that $(a_iz^i)$ is \low on $\{z\colon |z|<b\}$. We are finished, if
  we can show that $\sum_{i=0}^\infty a_iz^i$ is \F--convergent on
  $\{z\colon |z|<b\}$.

  For this we find a \low function $m(e)$ such that
  $|\sum_{i=n}^\infty a_iz^i|<\inv{e}$ for all $n\geq m(e)$ and
  $|z|\leq b-\inv{e}$. Since
  $\inv{\rho}=\limsup_{n\to\infty}\sqrt[n]{|a_n|}$, there is an $N$
  such that $\sqrt[n]{|a_n|}<\inv{b}$ for all $n> N$, or, $|a_nb^n$|<1
  for all $n>N$.

  We show that $\sum_{i=n}^\infty |a_i|x^i<\inv{e}$
  for all $n>\max(N,b^2e^3)$ and $x\in[0,b-\inv{e}]$:
  \begin{align*}
    \sum_{i=n}^\infty |a_i|x^i&=
    \Bigl(\frac{x}{b}\Bigr)^n
    \sum_{i=0}^\infty\bigl|a_{n+i}b^{n+i}\bigr|\Bigl(\frac{x}{b}\Bigr)^i
    \leq\Bigl(\frac{x}{b}\Bigr)^n
    \sum_{i=0}^\infty\Bigl(\frac{x}{b}\Bigr)^i\\
    &\leq\Bigl(1-\inv{be}\Bigr)^n\inv{1-\frac{x}{b}}
    \leq\Bigl(\inv{1+\inv{be}}\Bigr)^nbe\leq\inv{1+\frac{n}{be}}be\\
    &\leq \frac{(be)^2}{n}<\inv{e}
  \end{align*}
\end{proof}
\begin{bemerkung}
  If $0<b_0<b_1$ are \F--reals and $(a_ib_1^i)$ is an
  \F--sequence, then also $(a_ib_0^i)$ is an \F--sequence.
\end{bemerkung}
\begin{proof}
  $\Bigl(\dfrac{b_0}{b_1}\Bigr)^i$ is an \F--sequence.
\end{proof}

\begin{definition}
  Let $A$ be a compact subset of $\R^N$. We call a function
  $F:A\to\R^M$ to be in \F if there are \F--functions $d:\N\to\N$ and
  $f:\Q^N\times\N\to\Q^M$ such that for all $e\in\Np$ and all
  $a\in\Q^N$ and $x\in A$
  \begin{equation*}
    |x-a|<\inv{d(e)} \to |F(x)-f(a,e)| < \inv{e}.
  \end{equation*}
\end{definition}
Let $F:A\to\R^M$ be defined on the compact set $A$. The following is
easy to see:
\begin{enumerate}
\item If $F$ can be extended to an \F--function defined on an open
  set, then $F$ is in \F.
\item If $F$ is in \F, then all restrictions to open subsets of $A$
  belong to \F.
\end{enumerate}

\begin{lemma}\label{l:low-coeff}
  Let $F(z)=\sum_{i=0}^\infty a_iz^i$ be a complex power series with
  radius of convergence $\rho$. Assume that for some \F--real
  $b<\rho$, $F(z)$ restricted to $\{z\colon |z|\leq b\}$ is in
  \F. Then the sequence $(a_ib^i)$ belongs to \F.
\end{lemma}

\begin{proof}
  We have
  \begin{align*}
    a_n&=\inv{2\pi\ii}\int_{|z|=b}\frac{F(z)}{z^{n+1}}\dd z
    =\inv{2\pi\ii}\int_{|z|=1}\frac{F(bz)}{(bz)^{n+1}}\dd(bz)\\
    &=\inv{2\pi\ii}\inv{b^n}\int_{|z|=1}
    \frac{F(bz)}{z^{n+1}}\dd z
  \end{align*}
  The integral can be computed as
  \begin{align*}
  &\int_0^{2\pi}\frac{F(b\exp(x\ii))}{\exp(x\ii)^{n+1}}
    (\ii\exp(x\ii))\dd
    x\\ &=\ii\int_0^{2\pi}F(b\exp(x\ii))\exp(-nx\ii)\dd x
  \end{align*}
  An application of Theorem \ref{s:int} yields that $y\mapsto
  \ii\int_0^{2\pi}F(b\exp(x\ii))\exp(-yx\ii)\dd x$ is a \low function
  from $\R$ to $\C$. The lemma follows from this by an application of
  Lemma~\ref{l:f-function-sequence}.
\end{proof}
\begin{lemma}
  A sequence $(x_n)\in\C$ is in \F if there is an \F--function
  $G:\N^2\to\Q^2$ with $|x_n-G(n,e)|<1/e$ for all $n,e\in\Np$.
\end{lemma}
\begin{proof}
  This is clear from the definitions.
\end{proof}

\begin{lemma}[Speed-Up Lemma]
  Suppose $(a_n)\in\C$ is a bounded sequence and that $0<b<1$ is an
  \F--real. Then $(a_nb^n)$ is an \F--sequence if $(a_nb^{2n})$
  is.
\end{lemma}

\begin{proof}
  We may assume that $|a_n|<1$ for all $n$.  Let $f:\N^2\to\Q^2$ be in
  \F such that $|a_nb^{2n}-f(n,e)|<\inv{e}$.
  \[y_e=\frac{\ln(e)}{\ln(b^{-1})},\quad(e=1,2,\ldots)\]
  is an \F--sequence of reals. So there is an \F--function
  $h:\N\to\Q$ such that $|y_e-h(e)|<1$ for all $e\in\Np$. We fix also
  a natural number $B\geq b^{-1}$.\\

  We want to define an \F--function $G:\N^2\to\Q^2$  with
  \[\bigl|a_nb^n-G(n,e)\bigr|<\inv{e}\]
  for all $n,e\in\Np$. Let $n,e$ be given.  We distinguish two cases:\\

  \noindent Case 1: $h(e)<n-1$. We have then
  $y_e<n$, which implies \[|a_nb^n|<|b^n|<\inv{e}.\]
  So we set $G(n,e)=0$.\\

  \noindent Case 2: $h(e)\geq n-1$. Then $y_e\geq n-2$, which means
  \[b^{2-n}=\exp\bigl(\ln(b^{-1})(n-2)\bigr)\leq e.\]
  We can now apply Lemma \ref{l:bounded_exp} and compute $b^{2-n}$ as
  a \low function of $\ln(b^{-1})(n-2)$ and $e$.  It follows that we
  can compute $b^{-n}=b^{2-n}b^{-2}$ as an \F--function of $n$ and
  $e$. So also $G(n,e)=f(n,B^2e^2)b^{-n}$ is an \F--function of $n$
  and $e$, and we have
  \[\bigl|a_nb^n-G(n,e)\bigr|= \bigl|a_nb^{2n}-f(n,B^2e^2)\bigr|\cdot b^{-n}<
  \frac{b^{-n}}{B^2e^2}
  \leq\frac{b^{2-n}}{e^2}\leq\inv{e}.\]
\end{proof}
\begin{folgerung}\label{f:speedup}
  Let $F(z)=\sum_{i=0}^\infty a_iz^i$ be a complex power series with
  radius of convergence $\rho$. If $F$ is in \F on some closed
  subdisc of $\{z\colon |z|<\rho\}$, it is in \F on all closed
  subdiscs.
\end{folgerung}
\begin{proof}
  Let $r$ be any positive rational number smaller than $\rho$. Choose
  a rational number $s$ between $r$ and $\rho$.  We then have
  $|a_ns^n|<1$ for almost all $n$. By assumption there is an $N$ such
  that $F$ is \F on $\{z\colon |z|\leq sb_N\}$, where
  \[b_N=\Bigl(\frac{r}{s}\Bigr)^{2^N}.\]
  This implies by Lemma \ref{l:low-coeff} that $((a_ns^n)b_N^n)$ is in
  \F. If $N>0$, the Speed-Up Lemma shows that $((a_ns^n)b_{N-1}^n)$
  is in \F. Continuing this way we conclude that
  $(a_ns^nb_0^n)=(a_nr^n)$ is in \F. So $F$ is \F on $\{z\colon
  |z|<r\}$ by Lemma \ref{l:sequence-low}.
\end{proof}
\begin{satz}\label{s:locally_F}
  Let $F$ be a holomorphic function, defined on an open domain
  $D\subset\C$. If $F$ is in \F on some non--empty open subset of
  $D$, it is in \F on every compact subset of $D$.
\end{satz}
\begin{proof}
  It is easy to see that one can connect any two
  rational\footnote{i.e.\ in $\Q^2$} points $a,b$ in $D$ by a chain
  $a=a_0,\ldots,a_n=b$ of rational points such that for every $i<n$,
  some circle $O_i=\{z\colon |z-a_i|<r_i\}$ contains $a_{i+1}$ and is
  itself contained in $D$.

  If $F$ is in \F on some open neighborhood of $a_0$, Corollary
  \ref{f:speedup} (applied to $F(z+a_0)$) shows that $F$ is in \F on
  any closed subdisc of $O_0$. So $F$ is in \F in some open
  neighborhood of $a_1$, etc. We conclude that all rational points of
  $D$ have an open neighborhood on which $F$ is \F. So, again by
  Corollary \ref{f:speedup}, $F$ is in \F on all closed discs with
  rational center contained in $D$. Since we can cover any compact
  subset of $D$ with a finite number of such discs, the theorem
  follows.
\end{proof}

We call all a holomorphic function which satisfies the condition of
Theorem~\ref{s:locally_F} to be \emph{locally} in \F.
\begin{folgerung}
  Let $F$ be a holomorphic function, defined on an open domain
  $D\subset\C$. Let $a$ be an \F--complex number in $D$ and let $b$ be
  a positive \F--real smaller than the radius of convergence of
  $F(a+z)=\sum_{i=0}^\infty a_nz^n$. Then $F$ is locally in \F if and
  only if $(a_nb^n)$ is an \F--sequence.\qed
\end{folgerung}
\begin{folgerung}
  Let $F$ be holomorphic on a punctured disk $D_\bullet=\{z\mid
  0<|z|<r\}$. Then the following holds:
  \begin{enumerate}
  \item\label{enum:pole} If $0$ is a pole of $F$ and $F$ is \F on
    some non--empty open subset of $D_\bullet$, then $F$ is \F on every
    proper punctured subdisc $D'_\bullet=\{z\mid0<|z|<r'\}$.
  \item\label{enum:essential} If $0$ is an essential singularity of
    $F$, $F$ is not \low on $D_\bullet$.
  \end{enumerate}
\end{folgerung}
\begin{proof}
  \ref{enum:pole}: Let $0$ be a pole of order $k$. Then $F(z)z^k$ is
  holomorphic on $D=\{z\:|z|<r\}$. By the theorem $F(z)z^k$ is \F on
  any disc $D'=\{z:|z|<r'\}$, $r'<r$. Since $z^{-k}$ is \low on
  $D'_\bullet$, $F$ is \F on $D'_\bullet$.\\

  \noindent\ref{enum:essential}: If $F$ would be \low on $D_\bullet$,
  the absolute value of $F$ on $\{z\mid0<|z|<\inv{e}\}$ would be
  bounded by a polynomial in $e$ (Lemma \ref{l:bounded}). So $0$ would
  be a pole of $F$.
\end{proof}
\begin{folgerung}
  Let $S=\{-n\mid n\in\N\}$ denote the set of poles of the Gamma
  function $\Gamma$. $\Gamma$ is \low on every set
  $\{z:|z|<r\}\setminus S$.
\end{folgerung}
$\Gamma$ cannot be \low on $\C\setminus S$ since $n!$ grows too fast.
We believe that $\Gamma$ is \el on $\C\setminus S$.

\begin{folgerung}
  The Zeta function $\zeta(z)$ is \low on every punctured disk
  $\{z\mid 0<|z-1|<r\}$.
\end{folgerung}
$\zeta$ cannot be \low on $\C\setminus\{1\}$, since $\infty$ is an
essential singularity. But $\zeta$ may be \el on $\C\setminus\{1\}$.
\begin{folgerung}
  The set $\C_\F=\R_\F[\ii\,]$ of \F--complex numbers is
  algebraically closed and closed under $\ln(z)$, $\exp(z)$,
  $\Gamma(z)$ and $\zeta(z)$.
\end{folgerung}
Note that $\R_\F[\ii\,]$ is algebraically closed since $\R_\F$ is real
closed by Corollary~\ref{f:real_closed}.

If $a_0, a_1,\ldots$ are $\Q$--linearly independent algebraic numbers,
the exponentials $\exp(a_0), \exp(a_1).\ldots$ are \low and
algebraically independent by the Lindemann--Weierstraß Theorem.  So
the field of \low complex numbers has infinite transcendence degree.\\

\begin{bemerkung}
  If a holomorphic function belongs to \F, then also its derivative
  belongs to \F.
\end{bemerkung}
\begin{proof}
  This follows from the formula
  \[F'(z_0)=\inv{2\pi\ii}\int_{|z-z_0|=r}
  \frac{F(z)}{(z-z_0)^2}\dd z.\]
\end{proof}
\begin{bemerkung}
  Let $F$ be holomorphic function defined on an open domain, which
  maps some \F-complex number $a$ to an \F--complex number
  $F(a)$. Then $F$ is locally in \F if and only if its derivative is.
\end{bemerkung}
\begin{proof}
  This follows from the fact that if $a_0$ and $b$ are in \F, then
  $(a_nb^n)_{n\geq 0}$ is an \F--sequence if and only if
  $(na_nb^{n-1})_{n\geq 1}$ is an \F-sequence.
\end{proof}
Pour-El and Richards (\cite{pour-el}) have shown that this is not true
for functions of the reals: There is a recursive $C^1$--function with
non--recursive derivative.
\begin{bemerkung}
  Let $F$ be a non-constant holomorphic function which belongs to \F.
  Then $a$ is in \F if and only if $F(a)$ is in \F.
\end{bemerkung}
\begin{proof}
  Assume first that $F'(a)$ is not zero. Then we can find a small open
  rectangle $O$ with \low endpoints which contains $a$ and such that
  $F$ defines a homeomorphism between $U$ and $F(U)$. We can make $U$
  small enough such that $F\restriction U$ satisfies the conditions of
  Theorem~\ref{s:ift} (see also Lemma~\ref{l:gamma}). Then the inverse
  of $F\restriction U$ is in \F and maps $F(a)$ to $a$.

  If $F'(a)=0$, let $n$ be minimal such that $F^{(n+1)}(a)$ is not
  zero. Since $F^{(n)}$ is in \F and $F^{(n)}(a)=0$, the above shows
  that $a$ is in \F.
\end{proof}

\noindent
We close with two more examples.\\

\beispiel{} The function $\exp(\inv{z})$ is \low on every annulus
$\{z\mid r<|z|\,\}$, $r>0$, but not on $\C\setminus\{0\}$.\\

\beispiel{} There is a \low function $f:\N\to\{0,1\}$ such that
$n\mapsto f(2^n)$ is not \low. Consider the series
$F(z)=\sum_{n=0}^\infty a_nz^n$, where $a_n=f(2^n)$. $F$ is
holomorphic on $\E$ and \low on every compact subset of $\E$, but the
sequence $(a_n)$ is not \low.

\subsection*{Added in proof (26.9.2010)}
The definition of a (primitive) recursive function $F:\R\to\R$ given
by E. Specker in \cite{specker} is equivalent to the following: There
are (primitive) recursive functions $d:\N\to\N$ and
$f:\Q\times\N\to\Q$ such that for all $e\in\Np$ and all $a\in\Q$ and
$x\in\R$
\[ |x-a|<\inv{d(e)} \to |F(x)-f(a,e)| < \inv{2^e}.\]
Locally this agrees with our Definition \ref{d:f-function}.

\bibliographystyle{plain}
\bibliography{periods}
\vspace{4mm}

\begin{small}
\noindent\parbox[t]{15em}{
Katrin Tent,\\
Mathematisches Institut,\\
Universit\"at M\"unster,\\
Einsteinstrasse 62,\\
D-48149 M\"unster,\\
Germany,\\
{\tt tent@math.uni-muenster.de}}
\hfill\parbox[t]{18em}{
Martin Ziegler,\\
Mathematisches Institut,\\
Albert-Ludwigs-Universit\"at Freiburg,\\
Eckerstr. 1,\\
D-79104 Freiburg,\\
Germany,\\
{\tt ziegler@uni-freiburg.de}}

\end{small}
\end{document}